\documentclass{amsart}
\usepackage{amsmath,amsthm,enumerate,mathrsfs,graphicx,makecell,amssymb,booktabs,color}

\newtheorem{theorem}{Theorem}

\DeclareMathOperator{\diam}{diam}

\title{Prevalent uniqueness in ergodic optimisation}
\author{Ian D. Morris}
\thanks{
The author thanks Jairo Bochi for numerous helpful conversations and discussions. This research was partially supported by Leverhulme Trust Research Project Grant RPG-2016-194 and by FONDECYT 1180371.}
\address{Mathematics Department, University of Surrey, Guildford GU2 7XH, United Kingdom.}
\email{i.morris@surrey.ac.uk}
\begin{document}
\begin{abstract}
One of the fundamental results of ergodic optimisation asserts that for any dynamical system on a compact metric space $X$ and for any Banach space of continuous real-valued functions on $X$ which embeds densely in $C(X)$ there is a residual set of functions in that Banach space for which the maximising measure is unique. We extend this result by showing that this residual set is additionally prevalent, answering a question of J. Bochi and Y. Zhang.
\end{abstract}
\maketitle

In ergodic optimisation one considers a dynamical system on a compact metric space $X$ and asks, for a given function $f \colon X \to \mathbb{R}$, which invariant measures of the dynamical system give the largest ergodic average to $f$. We will call such measures \emph{maximising measures} for $f$ throughout this article. In this article we will assume that the dynamical system on $X$ is in discrete time and is generated by a continuous mapping $T \colon X\to X$, but the ideas of this article may be straightforwardly adapted to broader contexts such as flows and the actions of amenable groups. 

The existence of at least one maximising measure for a given continuous function $f\colon X \to \mathbb{R}$ is guaranteed by elementary compactness considerations. Given any nonempty compact $T$-invariant set $K \subset X$ it is rather trivial to construct a continuous potential $f \colon X \to \mathbb{R}$ whose maximising measures are precisely the invariant measures supported on $K$, simply by defining $f(x):=-\inf_{y \in K}d(x,y)$. In the case where $X$ is a manifold it is trivial to modify this construction so as to give $f$ the same property whilst also ensuring that $f$ has $C^k$ or smooth regularity. Using more delicate and abstract arguments it is also possible to construct, for any given ergodic invariant measure $\mu$ on $X$, a continuous function $f \colon X \to \mathbb{R}$ whose unique maximising measure is precisely $\mu$, as was undertaken in \cite{Je06b}; but there are, in general, nontrivial obstructions to constructing such a function $f$ in higher regularity classes such as the class of H\"older continuous functions, since the allowed set of maximising measures for such functions is in general constrained: see \cite[\S5]{Je19} for a broad discussion of this point, as well as such articles as \cite{Bo01,BoJe02,Mo07} for various theorems of the same general character. For the above reasons the study of maximising measures has therefore focused on understanding the maximising measures of \emph{typical} elements of some space of functions $X \to \mathbb{R}$, which is in almost all cases a Banach space of continuous functions. Since ergodic optimisation is entirely trivial when $T$ admits only a single invariant measure, and is only slightly less trivial when the number of invariant measures is finite, research has furthermore tended to focus on systems in which the spectrum of invariant measures is extremely large: with only very few exceptions, researchers in ergodic optimisation have always assumed the dynamical system on $X$ to be either an expanding or hyperbolic transformation of a manifold, a subshift of finite type,  or a hyperbolic flow.  In practice the study of typical H\"older continuous, Lipschitz continuous or $C^1$ functions $f\colon X \to \mathbb{R}$ has been found most fruitful.

In order to study general questions of this kind -- ``Given a hyperbolic dynamical system $T\colon X \to X$ and a Banach space $\mathfrak{X}$ of real-valued continuous functions $X \to \mathbb{R}$, what can be said about the maximising measures of typical elements of $\mathfrak{X}$?'' -- one must of course be equipped with an understanding of what is meant by ``typical''. From the early work of G. Yuan and B.R. Hunt (in \cite{YuHu99}) onwards research has focused on the standard \emph{topological} notion of typicality, namely that of a \emph{residual set}. We recall that if $\mathfrak{X}$ is a complete metrisable space then a subset of $\mathfrak{X}$ is called a \emph{$G_\delta$ set} if it is equal to the intersection of a countable family of open sets, and is called \emph{residual} if it is both $G_\delta$ and dense. By Baire's theorem every residual subset of such a space $\mathfrak{X}$ is dense, and the intersection of countably many residual sets is also residual. If $\mathfrak{X}$ is a Banach space it is easily also seen that the class of residual subsets of $\mathfrak{X}$ is closed with respect to translation and scalar multiplication. Equipped with this definition, research in ergodic optimisation has historically focused on attempting to show that for appropriate dynamical systems on compact metric spaces $X$ and appropriate Banach spaces $\mathfrak{X}$ of continuous functions $X\to \mathbb{R}$ there exists a residual subset of $\mathfrak{X}$ such that for every function in that residual set, the set of maximising measures of the function has some desired property such as being a singleton set, containing only measures of zero entropy, or containing a measure which is supported on a periodic orbit (see for example \cite{Bo01,BoJe02,Co16,CoLoTh01,HuLiMaXuZh19a,HuLiMaXuZh19b,Mo08,Mo10,QuSi12}). In one of the profoundest achievements in ergodic optimisation so far, a long-standing question of Yuan and Hunt -- namely, is it the case that if $T \colon X \to X$ is expanding or hyperbolic, a typical $C^1$ function has a periodic orbit as its sole maximising measure? -- has recently been satisfactorily answered in this specific understanding of the word ``typical'' (see \cite{Co16,HuLiMaXuZh19a}).

There is however a second notion of a set of ``typical'' elements of a Banach space which is rather more probabilistic or measure-theoretic in nature. If $\mathfrak{X}$ is a topological vector space equipped with a complete metric, a subset $\mathcal{P}$ of $\mathfrak{X}$ is called \emph{prevalent} if there exists a compactly supported Borel probability measure $m$ on $\mathfrak{X}$ such that for every $x \in \mathfrak{X}$ the translated set $x+\mathcal{P}$ has full measure with respect to $m$. As is the case for residual sets, the class of prevalent sets is closed with respect to translation and scalar multiplication, every prevalent set is dense, and every countable intersection of prevalent sets is prevalent; proofs of these statements may be found in \cite{HuSaYo92}. If $\mathfrak{X}$ is finite-dimensional then its prevalent subsets are precisely its subsets with full Lebesgue measure. As such the notions of prevalent and residual are orthogonal: a subset of $\mathfrak{X}$ may be residual but not prevalent, or prevalent without being residual. This makes it natural to ask the following question: for expanding or hyperbolic dynamical systems $T \colon X \to X$ and $C^1$ functions $X \to \mathbb{R}$, is it the case that a \emph{prevalent} set of functions has a periodic orbit as its sole maximising measure? We remark that measure-theoretic notions of typicality were emphasised in the very earliest works on ergodic optimisation (see \cite{HuOt96a,HuOt96b}) and the literature on the subject steered in the direction of topological notions of typicality only subsequently.

At the time of writing, the entire body of literature on the ergodic optimisation of prevalent sets of functions consists of the single article \cite{BoZh16}, in which a space of functions over shifts of finite type is created for the express purpose of proving that it admits a prevalent set of functions whose maximising measures are all supported on periodic orbits. The purpose of this note is to initiate the study of prevalent properties of maximising measures on more general function spaces. In support of this goal we prove a prevalent version of one of the fundamental nontrivial results in ergodic optimisation: generic uniqueness of the maximising measure. We hope that this will serve to stimulate research in ergodic optimisation in the direction of understanding the maximising measures of prevalent sets of functions, with the ultimate goal of extending the results of works such as \cite{Co16,HuLiMaXuZh19a,HuLiMaXuZh19b,Mo08} to the prevalent context.

In order to state and prove our result we require some additional definitions and notation. If $T \colon X \to X$ is a continuous transformation of a compact metric space, let $\mathcal{M}_T$ denote the set of all Borel probability measures on $X$ which are $T$-invariant. Via the Riesz representation theorem we may identify $\mathcal{M}_T$ with a subset of the dual space $C(X)^*$ in the weak-* topology. We recall that the weak-* topology on $\mathcal{M}_T$ is characterised as the smallest topology such that the functional $\mu \mapsto \int f\,d\mu$ is continuous for every $f \in C(X)$. In this topology $\mathcal{M}_T$ is compact and metrisable; it is also nonempty. For each $f \in C(X)$ we define $\beta(f):=\sup_{\mu \in \mathcal{M}_T} \int f\,d\mu$ and let $\mathcal{M}_{\max}(f)$ denote the set of all $\mu \in \mathcal{M}_T$ such that $\int f\,d\mu=\beta(f)$. By compactness and continuity $\mathcal{M}_{\max}(f)$ is nonempty for every $f \in C(X)$. 

Early in the development of ergodic optimisation it was noted that for many natural Banach spaces $\mathfrak{X}$ consisting of continuous functions $X \to \mathbb{R}$, there is a residual set of $f$ for which $\mathcal{M}_{\max}(f)$ consists of a single measure (see for example \cite{Bo01,CoLoTh01}). A very general expression of this fact is given by the following theorem of O. Jenkinson, proved in \cite{Je06a}, which is itself a development of an earlier result of Contreras, Lopes and Thieullen in \cite{CoLoTh01}: if $T \colon X \to X$ is a continuous transformation of a compact metric space, and $\mathfrak{X}$ a real topological vector space of continuous functions which embeds densely and continuously in $C(X)$, then $\mathfrak{X}$ contains a residual set of functions $f$ such that $\mathcal{M}_{\max}(f)$ is a singleton. It was asked by J. Bochi and Y. Zhang in \cite[\S6]{BoZh16} whether the uniqueness of the maximising measure also holds on a prevalent set of functions at a similar level of generality. The result which we prove in this note answers their question affirmatively:
\begin{theorem}\label{th:one}
Let $T\colon X \to X$ be a continuous transformation of a compact metric space. Let $\mathfrak{X}$ be a real vector space of continuous functions $X \to \mathbb{R}$ equipped with a complete metric with respect to which translation and scalar multiplication are both continuous, and such that $\mathfrak{X}$ is dense as a subset of $C(X)$. Then the set
\begin{equation}\label{eq:he_chomnk}
\left\{f \in \mathfrak{X} \colon \mathcal{M}_{\max}(f)\text{ is a singleton}\right\}
\end{equation}
is prevalent.
\end{theorem}
\begin{proof}[Proof of Theorem \ref{th:one}]
The fundamental idea of the proof is to show that for every $g \in \mathfrak{X}$, the set
\begin{equation}\label{eq:heftychonk}
\left\{f \in \mathfrak{X} \colon \left\{\int g\,d\mu \colon \mu \in \mathcal{M}_{\max}(f)\right\}\text{ is a singleton}\right\}\end{equation}
is prevalent. Equipped with this information it will follow that if $(g_n)_{n=1}^\infty$ is a sequence of elements of $\mathfrak{X}$ which is dense in $C(X)$, then the set
\[\bigcap_{n=1}^\infty \left\{f \in \mathfrak{X} \colon \left\{\int g_n\,d\mu \colon \mu \in \mathcal{M}_{\max}(f)\right\}\text{ is a singleton}\right\}\]
is also prevalent; but by density this is precisely
\[\left\{f \in \mathfrak{X} \colon \left\{\int g\,d\mu \colon \mu \in \mathcal{M}_{\max}(f)\right\}\text{ is a singleton for all }g \in C(X)\right\}\]
and this in turn is equal to the set defined in \eqref{eq:he_chomnk}. To establish the prevalence of the set \eqref{eq:heftychonk}
for fixed $g \in \mathfrak{X}$ we will show that for every $f \in \mathfrak{X}$, the intersection of the above set with the line $\{f+tg \colon t \in \mathbb{R}\}$ is precisely the set of points of differentiability of the function $t \mapsto \beta(f+tg)$, and the latter set has full one-dimensional Lebesgue measure. The proof therefore proceeds through a sequence of claims concerning the function $\mathbb{R} \to \mathbb{R}$ defined by $t \mapsto \beta(f+tg)$ where $f,g \in \mathfrak{X}$ are arbitrary. 

We first claim that if $f,g \in \mathfrak{X}$ are specified then $t \mapsto \beta(f+tg)$ is continuous and is differentiable Lebesgue almost everywhere. This may be seen in two ways. On the one hand we by definition have
\[\beta(f+tg)=\sup_{\mu \in \mathcal{M}_T} \int f+tg\,d\mu\]
for all $t \in \mathbb{R}$, and this relation expresses $t \mapsto \beta(f+tg)$ as the pointwise supremum of a family of linear maps $\mathbb{R} \to \mathbb{R}$. Such a function is necessarily convex, and hence is continuous and differentiable Lebesgue a.e. by standard results from convex analysis (see e.g. \cite[Theorem 25.5]{Ro70}). This proves the claim. Alternatively, if $t_1,t_2 \in \mathbb{R}$ are arbitrary then for every $\mu \in \mathcal{M}_T$ we have
\begin{align*}\int f+t_1g\,d\mu &=  \int f + t_2  g\,d\mu +\int (t_1-t_2)g\,d\mu\\
& \leq \int f + t_2  g\,d\mu+ |g|_\infty |t_1-t_2|\leq \beta(f+t_2g)+|g|_\infty|t_1-t_2|\end{align*}
and by taking the supremum with respect to $\mu$ and rearranging we obtain
\[\beta(f+t_1g) -\beta(f+t_2 g) \leq |g|_\infty |t_1-t_2|.\]
By symmetry it follows that
\[|\beta(f+t_1g)-\beta(f+t_2g)| \leq |g|_\infty |t_1-t_2|\]
for all $t_1,t_2 \in \mathbb{R}$ and so $t \mapsto \beta(f+tg)$ is uniformly Lipschitz continuous. The differentiability of the function almost everywhere follows by Rademacher's theorem and this gives an alternative proof of the claim.

Our second claim is that the inequality
\[\sup_{\mu \in \mathcal{M}_{\max}(f)}\int g\,d\mu \leq \frac{\beta(f+\tau g)-\beta(f)}{\tau}\leq \sup_{\mu \in \mathcal{M}_{\max}(f+\tau g)}\int g\,d\mu \]
is valid for all $f,g \in \mathfrak{X}$ and all real $\tau>0$. Indeed, to obtain the first inequality let $\mu \in \mathcal{M}_{\max}(f)$ be arbitrary; we have
\[ \beta(f)+\tau \int g\,d\mu = \int f+\tau g\,d\mu \leq \beta(f+\tau g)\]
and by rearrangement
\[\int g\,d\mu \leq \frac{\beta(f+\tau g)-\beta(f)}{\tau} \]
so that the first claimed inequality follows by taking the supremum with respect to $\mu$. Similarly if $\mu \in \mathcal{M}_{\max}(f+\tau g)$ is arbitrary then
\[\beta(f+\tau g)=\int f\,d\mu + \tau \int g\,d\mu \leq \beta(f)+\tau\int g\,d\mu\]
and a similar rearrangement gives
\[\frac{\beta(f+\tau g)-\beta(f)}{\tau}\leq \int g\,d\mu.\]
We again take the supremum with respect to $\mu$ to complete the proof of the claim.

We thirdly claim that if $f,g \in \mathfrak{X}$ are arbitrary then
\[\limsup_{\tau \searrow 0} \sup_{\mu \in \mathcal{M}_{\max}(f+\tau g)}\int g\,d\mu \leq \sup_{\mu \in \mathcal{M}_{\max}(f)}\int g\,d\mu.\]
Fixing $f$ and $g$, choose a decreasing sequence $(\tau_n)_{n=1}^\infty$ of positive real numbers converging to $0$ such that
\[\lim_{n\to\infty} \sup_{\mu \in \mathcal{M}_{\max}(f+\tau_ng)}\int g\,d\mu=\limsup_{\tau \searrow 0} \sup_{\mu \in \mathcal{M}_{\max}(f+\tau g)}\int g\,d\mu.\]
For each $n \geq 1$ the set $\mathcal{M}_{\max}(f+\tau_n g)$ is a closed (hence compact) nonempty subset of $\mathcal{M}_T$ and so we may choose $\nu_n \in \mathcal{M}_{\max}(f+\tau_ng)$ such that
\[\int g\,d\nu_n = \sup_{\mu \in \mathcal{M}_{\max}(f+\tau_ng)}\int g\,d\mu.\]
Using the compactness and metrisability of $\mathcal{M}_T$ choose a subsequence $(\tau_{n_k})_{k=1}^\infty$ such that $\lim_{k \to \infty} \nu_{n_k} \in \mathcal{M}_T$ exists, and call this limit $\nu$. We have
\[\int f\,d\nu = \lim_{k \to \infty} \int f\,d\nu_{n_k} = \lim_{k \to \infty} \beta(f+\tau_{n_k}g)=\beta(f)\]
(where we have used the continuity result from the first claim) so in particular $\nu \in \mathcal{M}_{\max}(f)$. It follows that 
\begin{align*}\sup_{\mu \in \mathcal{M}_{\max}(f)}\int g\,d\mu \geq \int g\,d\nu &= \lim_{k \to \infty} \int g\,d\nu_{n_k}\\
& =\lim_{k \to \infty} \sup_{\mu \in \mathcal{M}_{\max}(f+\tau_{n_k}g)}\int g\,d\mu\\
&= \limsup_{\tau \searrow 0} \sup_{\mu \in \mathcal{M}_{\max}(f+\tau g)}\int g\,d\mu\end{align*}
as required to prove the claim.

We fourthly claim that if $f,g \in\mathfrak{X}$ and $t \in \mathbb{R}$ are fixed then the derivative 
\[\frac{d}{dt} \beta(f+tg)\Big|_{t=t_0}=\lim_{\tau \to 0} \frac{\beta(f+(t_0+\tau)g)-\beta(f+t_0g)}{\tau}\]
exists if and only if $\{\int g\,d\mu \colon \mu \in \mathcal{M}_{\max}(f+t_0g)\}$ is a singleton set, and if the derivative exists then it is equal to the sole element of that singleton. Clearly by considering $f+t_0g$ in place of $f$ we may without loss of generality reduce to the case $t_0=0$. Applying this reduction and considering the effect of the second and third claims on $f$ and $g$ we find that
\begin{align*}\sup_{\mu \in \mathcal{M}_{\max}(f)}\int g\,d\mu &\leq \liminf_{\tau \searrow 0} \frac{\beta(f+\tau g)-\beta(f)}{\tau}\\
&\leq  \limsup_{\tau \searrow 0} \frac{\beta(f+\tau g)-\beta(f)}{\tau}\\
&\leq \limsup_{\tau \searrow 0}  \sup_{\mu \in \mathcal{M}_{\max}(f+\tau g)}\int g\,d\mu \leq \sup_{\mu \in \mathcal{M}_{\max}(f)}\int g\,d\mu\end{align*}
so that
\[\lim_{\tau \searrow 0}  \frac{\beta(f+\tau g)-\beta(f)}{\tau} = \sup_{\mu \in \mathcal{M}_{\max}(f)} \int g\,d\mu\]
for all $f,g \in \mathfrak{X}$. Considering instead $f$ and $-g$ the same argument yields
\begin{align*}\lim_{\tau \nearrow 0} \frac{\beta(f+\tau g)-\beta(f)}{\tau} &= \lim_{\tau \searrow 0}  \frac{\beta(f-\tau g)-\beta(f)}{-\tau}\\
& = -\sup_{\mu \in \mathcal{M}_{\max}(f)} \int (-g)\,d\mu= \inf_{\mu \in \mathcal{M}_{\max}(f)} \int g\,d\mu.\end{align*}
Hence the derivative of $t \mapsto \beta(f+t g)$ exists at $t=0$ if and only if $\inf_{\mu \in \mathcal{M}_{\max}(f)} \int g\,d\mu$ and $\sup_{\mu \in \mathcal{M}_{\max}(f)} \int g\,d\mu$ are equal, and when the derivative exists it is equal to their common value. The claim follows easily.

We may now prove the theorem. Since $\mathfrak{X}$ is by hypothesis dense in $C(X)$, and since $C(X)$ is separable, we may choose a sequence $(g_n)_{n=1}^\infty$ of elements of $\mathfrak{X}$ which is dense in $C(X)$. For each $n \geq 1$ let $m_n$ be the Borel probability measure on $\mathfrak{X}$ given by defining one-dimensional Lebesgue measure on the interval $\{t g_n \colon t \in [0,1]\}$ in the obvious fashion. Consider the sets
\[\mathcal{P}_n:=\left\{f \in \mathfrak{X} \colon \left\{\int g_n\,d\mu \colon \mu \in \mathcal{M}_{\max}(f)\right\}\text{ is a singleton}\right\}\]
for $n \geq 1$. By the fourth claim, for every $f \in \mathfrak{X}$ the Lebesgue measure of the set
\begin{equation}\label{eq:oh-lawd-he-comin}\left\{t \in [0,1]\colon \left\{\int g_n\,d\mu \colon \mu \in \mathcal{M}_{\max}(f+tg_n)\right\}\text{ is a singleton}\right\}\end{equation}
is equal to the Lebesgue measure of the set of differentiability points of the function $t \mapsto \beta(f+tg_n)$ on $[0,1]$; but by the first claim, that set has full measure. Since the Lebesgue measure of the set \eqref{eq:oh-lawd-he-comin} is precisely $m_n(-f+\mathcal{P}_n)$ we conclude that every translate of $\mathcal{P}_n$ has full $m_n$-measure, which is to say that $\mathcal{P}_n$ is prevalent. 
By countable intersection the set
\[\mathcal{P}:=\left\{f \in \mathfrak{X} \colon \text{for all } n\geq 1, \left\{\int g_n\,d\mu \colon \mu \in \mathcal{M}_{\max}(f)\right\}\text{ is a singleton}\right\}\]
is also prevalent. If $f \in \mathcal{P}$ then by the density of $(g_n)$ in $C(X)$ we deduce that
\[\mathcal{P}=\left\{f \in \mathfrak{X} \colon \text{for all } g \in C(X), \left\{\int g\,d\mu \colon \mu \in \mathcal{M}_{\max}(f)\right\}\text{ is a singleton}\right\}\]
and so the latter set is prevalent. But it is clear that this set is precisely the set of all $f \in \mathfrak{X}$ such that $\mathcal{M}_{\max}(f)$ is a singleton set: if $\mathcal{M}_{\max}(f)$ is a singleton then so trivially is $\{\int g\,d\mu\colon \mu \in \mathcal{M}_{\max}(f)\}$, and if $\mathcal{M}_{\max}(f)$ is not a singleton then it contains two measures which, being distinct, must necessarily assign differing integrals to some function $g \in C(X)$. Hence the set
\[\left\{f \in \mathfrak{X} \colon \mathcal{M}_{\max}(f)\text{ is a singleton}\right\}\]
is equal to the prevalent set $\mathcal{P}$. The proof is complete.
\end{proof}

To close this note we make some remarks on the relationship between the proof of Theorem \ref{th:one} and the proof of residual unique ergodic optimisation given by Jenkinson in \cite{Je06a}. Foundational to both proofs is the observation that for fixed $f \in \mathfrak{X}$ the set $\mathcal{M}_{\max}(f)$ is a singleton if and only if $\{\int g\,d\mu \colon \mu \in \mathcal{M}_{\max}(f)\}$ is a singleton for every $g \in C(X)$, if and only if $\{\int g_n\,d\mu \colon \mu \in \mathcal{M}_{\max}(f)\}$ is a singleton for all $n \geq 1$ where $(g_n)$ is an a priori fixed dense subsequence of $C(X)$. To prove either theorem it is therefore sufficient to be able to show that for every $g \in \mathfrak{X}$ the set
\begin{equation}\label{eq:megachonker}\left\{f \in \mathfrak{X} \colon \left\{\int g\,d\mu \colon \mu \in \mathcal{M}_{\max}(f)\right\}\text{ is a singleton}\right\}\end{equation}
is either residual or prevalent as the situation requires. Jenkinson's proof can be interpreted as being based around the fact that $\{\int g\,d\mu \colon \mu \in \mathcal{M}_{\max}(f+tg)\}$ is a singleton at $t$ if and only if $t$ is a continuity point of the upper semi-continuous functional
\[t \mapsto\diam \left\{\int g\,d\mu \colon \mu\in \mathcal{M}_{\max}(f+tg)\right\},\]
and this principle is used in \cite{Je06a} to establish the denseness of the set \eqref{eq:megachonker}. This procedure may be adapted so as to prove residuality directly: by modifying Jenkinson's argument only slightly one may show that the set \eqref{eq:megachonker} is precisely the set of continuity points of the upper semi-continuous functional 
\[f \mapsto  \diam \left\{\int g\,d\mu \colon \mu\in \mathcal{M}_{\max}(f)\right\}\]
defined on all of $\mathfrak{X}$, and this completes the proof since the set of continuity points of an upper semi-continuous function is always residual. The strategy of Theorem \ref{th:one} differs from this retelling of Jenkinson's argument by considering instead the set of \emph{differentiability points} of the \emph{continuous} functional
\[t \mapsto \beta(f+tg)\]
for $f \in \mathfrak{X}$ and showing that it has full Lebesgue measure for fixed $f$ and $g$. We leave open to speculation the question of whether or not this principle might be adapted so as to establish the prevalence of other qualities of maximising measures by characterising them via the differentiability points of appropriate functionals.

\bibliographystyle{acm}
\bibliography{ergopt-biblio}
\end{document}